\documentclass[12pt]{amsart}

\usepackage{amsmath,amssymb,lscape}
\usepackage{graphicx}
\usepackage{epstopdf}
\usepackage{enumerate}

\newtheorem{theorem}{Theorem}[section]

\newtheorem{lemma}[theorem]{Lemma}

\theoremstyle{definition}
\newtheorem{definition}[theorem]{Definition}

\theoremstyle{remark}
\newtheorem{remark}[theorem]{Remark}

\newcommand\xqed[1]{%
  \leavevmode\unskip\penalty9999 \hbox{}\nobreak\hfill
  \quad\hbox{#1}}
\newcommand\tqed{\xqed{$\triangle$}}

\numberwithin{equation}{section}

\def\bR{{\mathbb {R}}}

\def\pB{{\mathcal B}}

\def\pL{{\mathcal L}}

\def\pM{{\mathcal M}}
\def\pP{{\mathcal P}}

\def\pS{{\mathcal S}}

\begin{document}

\title[Yoshikawa moves on marked graphs via Roseman's theorem]{Yoshikawa moves on marked graphs via Roseman's theorem}
\author{Oleg Chterental}
\email{oleg.chterental@mail.utoronto.ca}
\maketitle
\begin{center}\today\end{center}

\begin{abstract}
Yoshikawa \cite{Yo} conjectured that a certain set of moves on marked graph diagrams generates the isotopy relation for surface links in $\bR^4$, and this was proved by Swenton \cite{S} and Kearton and Kurlin \cite{KK}. In this paper, we find another proof of this fact for the case of $2$-links (surface links with spherical components). The proof involves a construction of marked graphs from branch-free broken surface diagrams, and a version of Roseman's theorem \cite{R} for branch-free broken surface diagrams of $2$-links.
\end{abstract}

\section{Introduction}

For a smooth oriented manifold $M$ we denote by ${\rm Diff}^+(M)$ the group of orientation-preserving diffeomorphisms of $M$ and by ${\rm Diff}_0^+(M)$ the path-component of the identity in ${\rm Diff}^+(M)$. We orient $\bR^4$ in the standard way. A \textbf{surface link} $L \subset \bR^4$ is a smooth submanifold diffeomorphic to a closed surface (i.e. compact with no boundary). We say two surface links $L_1$ and $L_2$ are \textbf{isotopic} if there is $f \in {\rm Diff}_0^+(\bR^4)$ with $f(L_1)=L_2$. A \textbf{$2$-link} is a surface link where each component is a $2$-sphere. Let $\pL$ be the set of surface links, $\pL^+$ the set of surface links with orientable components, and $\pL_0$ the set of $2$-links.

In Section \ref{sectrose} we review the definitions of generic projections, broken surface diagrams, Roseman moves, and Roseman's theorem. We prove Theorem \ref{thmbranch}, a version of Roseman's theorem for branch-free broken surface diagrams of $2$-links, stating that two isotopic branch-free broken surface diagrams of a $2$-link are related by a finite sequence of $Ro1$, $Ro2$, $Ro5^*$, $Ro7$, $Br1$ and $Br2$ moves (see Figures \ref{figrosemove} and \ref{figbranchmoves}). This theorem complements the results of Takase and Tanaka \cite{TT}, who find examples of isotopic branch-free broken surface diagrams of a $2$-link that are not related by $Ro1$, $Ro2$, $Ro5^*$ and $Ro7$ moves alone.

In Section \ref{sectgraph} we review marked graphs in $\bR^3$, marked graph diagrams in $\bR^2$, Yoshikawa moves on marked graph diagrams, $ab$-surfaces obtained from marked graphs, and prove some facts about marked graphs and $ab$-surfaces. In Section \ref{sectproof} we describe a relationship between branch-free broken surface diagrams and $ab$-surfaces, and provide another proof that two marked graph diagrams describe isotopic $2$-links if and only if they are related by a finite sequence of Yoshikawa moves (see Theorem \ref{thmmain}).

\section*{Acknowledgements}
The author wishes to thank Dror Bar-Natan, J. Scott Carter, Micha{\l} Jab{\l}onowski, Seiichi Kamada and Vassily O. Manturov for pointing out some inaccuracies and providing helpful references and comments.

\section{Broken surface diagrams and Roseman's theorem}\label{sectrose}

Definition \ref{defbsd} reviews generic projections, broken surface diagrams and Roseman moves. Roseman's theorem is stated in Theorem \ref{thmrose}. We use Lemma \ref{lembranch} to prove Theorem \ref{thmbranch}, a branch-free version of Roseman's theorem.

\begin{figure}
\begin{center}
\includegraphics[scale=1]{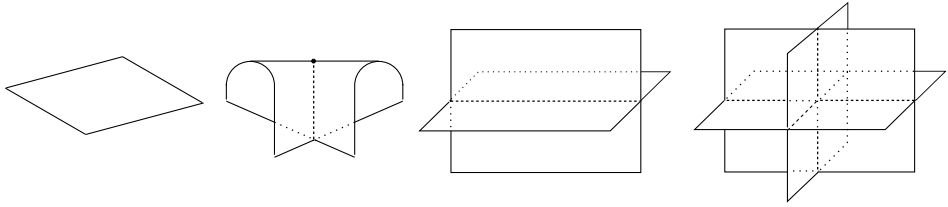}
\caption{Non-singular, branch, double and triple points in a generic projection.}\label{figgenproj}
\end{center}
\end{figure}

\begin{figure}
\begin{center}
\includegraphics[scale=1]{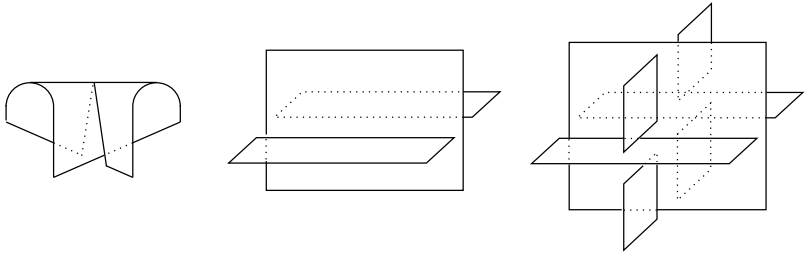}
\caption{Broken surface diagrams of branch, double and triple points.}\label{figbsd}
\end{center}
\end{figure}

\begin{definition}\label{defbsd}
References for the definitions we present here can be found for example in Carter, Kamada and Saito \cite{CKS}, Kamada \cite{Kam} and Roseman \cite{R}. Let $\pi:\bR^4 \rightarrow \bR^3$ be the projection $(x,y,z,u) \mapsto (x,y,z)$. If $L_1 \in \pL$ there is $L_2 \in \pL$ isotopic to $L_1$ such that a neighbourhood in $\bR^3$ of any point of $\pi(L_2)$ is one of the four possibilities in Figure \ref{figgenproj}. Such a projection $\pi(L_2)$ will be called \textbf{generic}.

Assume $L \in \pL$ is such that $D=\pi(L)$ is generic. Let ${\it sing}(D) \subset \bR^3$ be the set of branch, double and triple points. If $p \in {\it sing}(D)$ is a double point then there exist $x_1 \neq x_2 \in L$ with $\pi(x_1)=\pi(x_2)$ and these two points are ordered by the $u$-coordinates of $x_1$ and $x_2$. Similarly if $p \in {\it sing}(D)$ is a triple point then there exist $x_1 \neq x_2 \neq x_3 \in L$ with $\pi(x_1)=\pi(x_2)=\pi(x_3)$ and these three points are ordered by the $u$-coordinates of $x_1$, $x_2$ and $x_3$. The orderings extend by continuity to neighbourhoods of the points in $S$. The projection $D$ along with all of the above crossing information is called a \textbf{broken surface diagram}. Following Takase and Tanaka \cite{TT}, we refer to a broken surface diagram with no branch points as \textbf{branch-free}. We consider two broken surface diagrams \textbf{equivalent} if they differ by the action of ${\rm Diff}_0^+(\bR^3)$ and agree on crossing information. We may indicate the crossing information near a double-point by removing a neighbourhood of one of the pre-images with the convention that missing segments have greater $u$-coordinates as in Figure \ref{figbsd}. Figure \ref{figrosemove} depicts the Roseman moves on broken surface diagrams. Appropriate crossing information should be added for completeness. The move $Ro5^*$ is slightly different but equivalent modulo the $Ro1$ and $Ro2$ moves, to the usual $Ro5$ move. 

If $X$ is a set of surface links let $\pB(X)$ be the set of broken surface diagrams corresponding to links in $X$ with generic projections, and let $\pB^b(X) \subset \pB(X)$ be the subset of branch-free broken surface diagrams.
\tqed
\end{definition}

\begin{theorem}[Roseman \cite{R}]\label{thmrose}
If $D,D' \in \pB(\pL)$ represent isotopic surface links then $D$ and $D'$ are related by a finite sequence of the moves in Figure \ref{figrosemove} and equivalences in $\bR^3$.
\end{theorem}

\begin{figure}
\begin{center}
\includegraphics[scale=1]{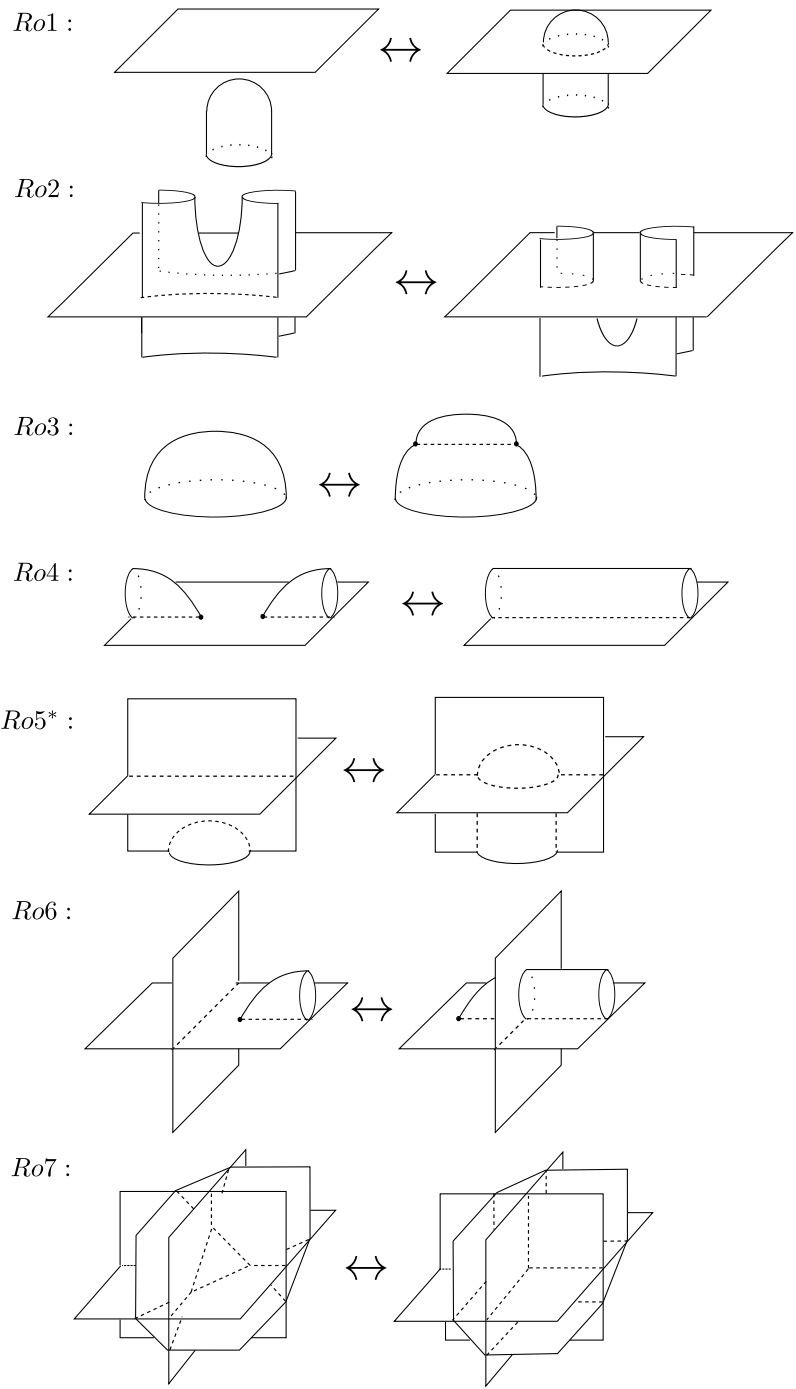}
\caption{Roseman moves on broken surface diagrams, with crossing information suppressed.}\label{figrosemove}
\end{center}
\end{figure}

\begin{figure}
\begin{center}
\includegraphics[scale=1]{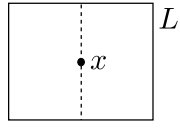}
\caption{The set $\pi^{-1}({\it sing}(\pi(L))) \cap L$ in the neighbourhood of a branch point $x$.}\label{figneighbranch}
\end{center}
\end{figure}

\begin{figure}
\begin{center}
\includegraphics[scale=1]{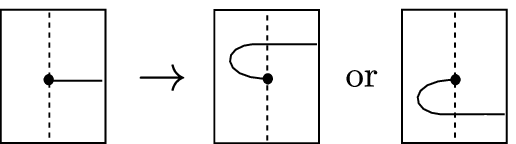}
\caption{Converting an invalid path to a valid path. See Definition \ref{defcancel} and Remark \ref{remdef}.}\label{figmakevalid}
\end{center}
\end{figure}

\begin{remark}
In the paper of Bar-Natan, Fulman and Kauffman \cite{BFK} there is a proof of the well-known fact that all spanning-surfaces of a classical link in $\bR^3$ are tube-equivalent. In that proof, link projections are used to construct Seifert surfaces via Seifert's algorithm, and the main result is deduced by observing how the constructed Seifert surfaces change when Reidemeister moves are performed on the link projection.

We wish to approach the proofs of Theorems \ref{thmbranch} and \ref{thmmain} in a similar manner. We will assign various structures (a set of branch-free broken surface diagrams in the case of Theorem \ref{thmbranch} and an $ab$-surface in the case of Theorem \ref{thmmain}) to a broken surface diagram and observe how these structures change when Roseman moves are performed on the broken surface diagram.
\end{remark}

\begin{definition}\label{defcancel}
Following Carter and Saito \cite{CS}, we define a function ${\it cancel}: \pB(\pL^+) \rightarrow \pP(\pB^b(\pL^+))$ (where $\pP(X)$ is the power set of $X$). Let $D=\pi(L)$ be a broken surface diagram corresponding to a generic projection for some $L \in \pL^+$ with components $L=L_1 \cup \ldots \cup L_r$. If $x \in L$ is such that $\pi(x)$ is a branch point, then a neighbourhood of $x \in \pi^{-1}({\it sing}(D)) \cap L$ looks as in Figure \ref{figneighbranch}. Let $D_i$ be the broken surface diagram obtained from $D$ by removing all components except $L_i$. Each $D_i$ has an equal number of positive and negative branch points, and any pair of positive and negative branch points can be cancelled using an appropriate sequence of $\overrightarrow{Ro6}$ moves followed by an $\overrightarrow{Ro4}$ move.

Specifically, assume $D_i$ has $2b_i \geq 0$ branch points and for $1 \leq i \leq r$ and $1 \leq j \leq b_i$ let $p_{i,j}, q_{i,j} \subset L_i$ be the distinct points such that $\pi(p_{i,j})$ and $\pi(q_{i,j})$ are positive and negative branch points in $D$. A pair $(T=\{\tau_i\}_{1 \leq i \leq r}, V=\bigcup_{1\leq i\leq r,1\leq j \leq b_i}v_{i,j})$ is \textbf{valid} if:
\begin{itemize}
\item[-] each $\tau_i$ is a permutation of $\{1,2,\ldots,b_i\}$,
\item[-] $V \subset L_i$ is a disjoint union of embedded oriented compact intervals $v_{i,j}$ with endpoints $\partial v_{i,j}=\{p_{i,j},q_{i,\tau_i(j)}\}$ and the orientation going from $p_{i,j}$ to $q_{i,\tau_i(j)}$,
\item[-] $V \cap (\pi^{-1}({\it sing}(D)) \cap L)$ is transverse in $L$ and $\pi(V)$ is embedded in $\bR^3$,
\item[-] after performing the sequence of $\overrightarrow{Ro6}$ moves pushing $p_{i,j}$ along $v_{i,j}$ through all intersections in $v_{i,j} \cap (\pi^{-1}({\it sing}(D)) \cap L)$, it is possible to perform a final $\overrightarrow{Ro4}$ move cancelling $p_{i,j}$ and $q_{i,\tau_i(j)}$.
\end{itemize}
The set ${\it cancel}(D)$ contains a broken surface diagram $E(T,V)$ for each valid pair $(T,V)$, obtained by actually performing on $D$ the aformentioned $\overrightarrow{Ro6}$ moves along each $v_{i,j}$ ending in an $\overrightarrow{Ro4}$ move cancelling the branch points $p_{i,j}$ and $q_{i,\tau_i(j)}$. If there are no branch points in $D$ we let ${\it cancel}(D)=\{D\}$.
\tqed
\end{definition}

\begin{remark}\label{remdef}
It is important to note that the final requirement for valid paths in Definition \ref{defcancel} is not superfluous. There exist invalid collections of simple paths satisfying the first three requirements. However if a particular path satisfies the first three requirements but not the fourth, then it can be made valid by the modification in Figure \ref{figmakevalid}. This is depicted on the level of broken surface diagrams in Carter and Saito \cite[Figure K and Figure M]{CS}.
\tqed
\end{remark}

\begin{remark}\label{remmoves}
Figure \ref{figpathmoves} describes $P$ moves on the valid pairs $(T,V)$ of Definition \ref{defcancel}. For each move, the diagram $D=\pi(L)$ remains fixed and the set $\pi^{-1}({\it sing}(D)) \cap L$ thus remains fixed as well. The moves $P2-P6$ correspond to an isotopy of $V$ rel $\partial V$ in $L$, interacting with the set $\pi^{-1}({\it sing}(D)) \cap L$. The moves $P1$ and $P7$ describe two types of interactions of two simple paths in $V$. The $P1$ move acts as a transposition on one of the permutations in $T$. While $P$ moves and valid pairs are somewhat adhoc objects, in Lemma \ref{lembranch} we will see that the moves $P1-P6$ correspond naturally to moves $Br1-Br6$ on broken surface diagrams, described in Figure \ref{figbranchmoves}. Note that each $Br$ move does in fact represent an isotopy in $\bR^4$, since it can be written in terms of Roseman moves, in particular $Ro4$ and $Ro6$ moves.
\tqed
\end{remark}

\begin{figure}
\begin{center}
\includegraphics[scale=1]{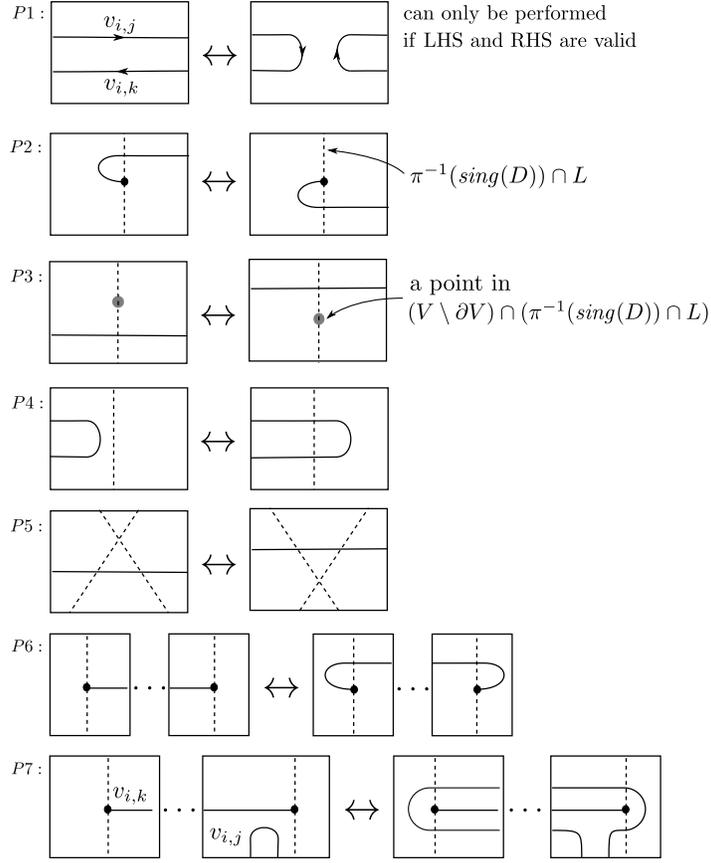}
\caption{Moves on valid pairs $(T,V)$ for some fixed broken surface diagram $D$. See Remark \ref{remmoves}.}\label{figpathmoves}
\end{center}
\end{figure}

\begin{figure}
\begin{center}
\includegraphics[scale=1]{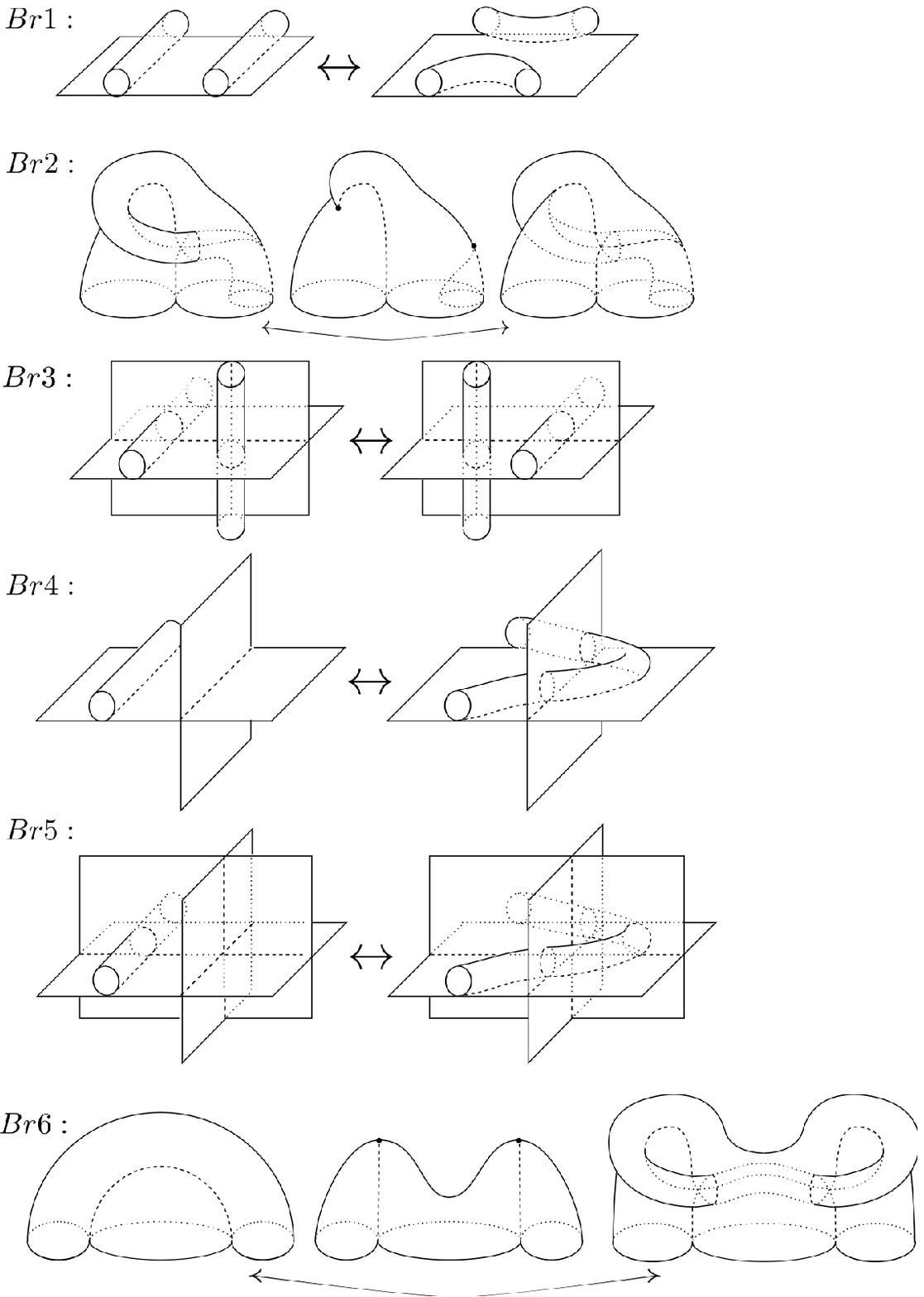}
\caption{$Br$ moves for broken surface diagrams, with crossing information suppressed.}\label{figbranchmoves}
\end{center}
\end{figure}

\begin{lemma}\label{lembranch}
\begin{enumerate}[a)]
\item\label{branch1} If $D \in \pB(\pL_0)$ then all valid pairs $(T,V)$ as in Definition \ref{defcancel} are related by the $P$ moves in Figure \ref{figpathmoves}.
\item\label{branch2} The $Br3$, $Br4$, $Br5$ and $Br6$ moves can be expressed in terms of $Ro1$, $Ro2$, $Ro5^*$, $Ro7$ and $Br1$ moves.
\item\label{branch3} If $D \in \pB(\pL_0)$ and $E_1,E_2 \in {\it cancel}(D)$ then $E_1$ and $E_2$ are related by the $Ro1$, $Ro2$, $Ro5^*$, $Ro7$, $Br1$ and $Br2$ moves in Figures \ref{figrosemove} and \ref{figbranchmoves}.
\item\label{branch4} If $D_1,D_2 \in \pB(\pL_0)$, $E_1 \in {\it cancel}(D_1)$, $E_2 \in {\it cancel}(D_2)$ and $D_1$ and $D_2$ are related by a Roseman move then $E_1$ and $E_2$ are related by a sequence of $Ro1$, $Ro2$, $Ro5^*$, $Ro7$, $Br1$ and $Br2$ moves.
\end{enumerate}
\end{lemma}
\begin{proof}
\begin{enumerate}[a)]
\item Assume we have two pairs $(T,V)$ and $(T',V')$ with the same set of permutations $T=T'=\{\tau_i\}_{1 \leq i \leq r}$ and let $V=\{v_{i,j}\}_{1\leq i\leq r,1\leq j \leq b_i}$ and $V'=\{v'_{i,j}\}_{1\leq i\leq r,1\leq j \leq b_i}$. If for some $i,j$ the pairs $v_{i,j}$ and $v'_{i,j}$ do not agree in small enough neighbourhoods of the points $p_{i,j}$ and $q_{i,\tau_i(j)}$ in $L$, then they can be made to agree using a $P6$ move, due to the fact that the pairs $(T,V)$ and $(T,V')$ satisfy the last property in Definition \ref{defcancel}. Thus we assume for each $i,j$ the paths $v_{i,j}$ and $v'_{i,j}$ agree in small enough neighbourhoods of the points $p_{i,j}$ and $q_{i,\tau_i(j)}$ in $L$. Since each component of $L$ is a sphere, there is an isotopy in $L$ taking $v_{i,j}$ to $v'_{i,j}$ relative to their common boundary $p_{i,j}$ and $q_{i,\tau_i(j)}$. During this isotopy, $v_{i,j}$ does not change in a small enough neighbourhood of its boundary. Such an isotopy can be accomplished using the $P2$, $P3$, $P4$, $P5$, and $P7$ moves. Thus after performing each isotopy for all choices of $i,j$ we get a sequence of $P$ moves taking $V$ to $V'$.  

Consider now the case of two pairs $(T,V)$ and $(T',V')$ with $T \neq T'$. It is enough to assume that all permutations in $T$ agree with their respective permutations in $T'$ except for some two permutations that differ by a transposition. We can induce an arbitrary transposition using a $P1$ move. If we wish to perform a $P1$ move between $v_{i,j}$ and $v_{i,k}$ we can use $P4$ moves to bring them into the exact form in the left-hand side of the $P1$ move. Now if the $P1$ move is valid (i.e. lifts to the $Br1$ move in Figure \ref{figbranchmoves}), we can perform the $P1$ move to induce a transposition. If the $P1$ move is not valid, it can be made valid after a $P6$ move along either of the paths. Thus we may induce arbitrary transpositions in the $\tau_i$'s using $P$ moves. 

Thus the $P$ moves suffice to connect all valid pairs $(T,V)$ in Definition \ref{defcancel}.

\item We leave it to the reader to verify that the $Br3$, $Br4$, and $Br5$ moves can be expressed in terms of $Ro1$, $Ro2$, $Ro5^*$ and $Ro7$ moves. Figure \ref{figbr6proof} expresses a $Br6$ move using $Ro1$, $Br1$ and $Br4$ moves.

\begin{figure}
\begin{center}
\includegraphics[scale=1]{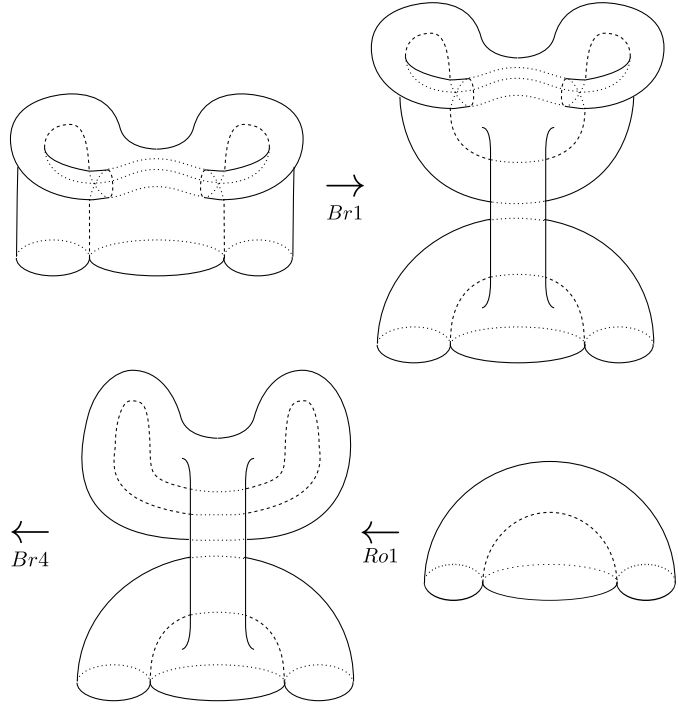}
\caption{Showing that a $Br6$ move can be realized with $Ro1$, $Br1$ and $Br4$ moves.}\label{figbr6proof}
\end{center}
\end{figure}

\item If $E_1,E_2 \in {\it cancel}(D)$ then by Lemma \ref{lembranch}\ref{branch1} their defining valid pairs are related by $P$ moves. The moves $P1$, $P2$, $P3$, $P4$ and $P5$ on valid pairs can be realized directly by the moves $Br1$, $Br2$, $Br3$, $Br4$ and $Br5$ respectively on broken surface diagrams. The move $P6$ can be realized by a sequence involving $Ro1$, $Ro2$, $Ro5^*$, $Ro7$ and $Br6$ moves. The $P7$ move can be realized by a sequence involving $Ro1$, $Ro2$, $Ro5^*$ and $Ro7$ moves. By Lemma \ref{lembranch}\ref{branch2}, the $Br3$, $Br4$, $Br5$ and $Br6$ moves can be expressed in terms of $Ro1$, $Ro2$, $Ro5^*$, $Ro7$ and $Br1$ moves, so we are done.

\item If $D_1$ and $D_2$ are related by an $Ro1$, $Ro2$, $Ro5^*$, or $Ro7$ move, then it is not difficult to see that there are diagrams $F_1 \in {\it cancel}(D_1)$ and $F_2 \in {\it cancel}(D_2)$ such that $F_1$ and $F_2$ are related by an $Ro1$, $Ro2$, $Ro5^*$, or $Ro7$ move. By Lemma \ref{lembranch}\ref{branch3}, $E_1$ is related to $F_1$ and $E_2$ is related to $F_2$ by a sequence of $Ro1$, $Ro2$, $Ro5^*$, $Ro7$, $Br1$ and $Br2$ moves. Thus there is a sequence of $Ro1$, $Ro2$, $Ro5^*$, $Ro7$, $Br1$ and $Br2$ moves relating $E_1$ to $E_2$.

If an $\overrightarrow{Ro3}$ move takes $D_1$ to $D_2$ then there are diagrams $F_1 \in {\it cancel}(D_1)$ and $F_2 \in {\it cancel}(D_2)$ with a $\overleftarrow{Ro1}$ move taking $F_1$ to $F_2$ (note that the two branch points created by the $\overrightarrow{Ro3}$ move can be paired in two obvious ways. One way can be cancelled by an $\overleftarrow{Ro1}$ move and the other can be cancelled by $\overleftarrow{Ro1}$, $Br1$ and $Br4$ moves, similar to Figure \ref{figbr6proof}). Thus by Lemma \ref{lembranch}\ref{branch3} again there is a sequence of $Ro1$, $Ro2$, $Ro5^*$, $Ro7$, $Br1$ and $Br2$ moves relating $E_1$ to $E_2$.

If $D_1$ and $D_2$ are related by an $Ro4$ or an $Ro6$ move, then ${\it cancel}(D_1) \subset {\it cancel}(D_2)$ or ${\it cancel}(D_2)\subset {\it cancel}(D_1)$ and we rely on Lemma \ref{lembranch}\ref{branch3} once more.
\end{enumerate}
\end{proof}

\begin{remark}
It should be stressed that our proof of Lemma \ref{lembranch}\ref{branch1} relies on the fact that in a sphere, any two simple paths with a common boundary are isotopic relative to their common boundary. \tqed
\end{remark}

Now we can prove a branch-free version of Roseman's theorem.

\begin{theorem}\label{thmbranch}
If $D_1,D_2 \in \pB^b(\pL_0)$ are related by a finite sequence of Roseman moves then they are related by a finite sequence of $Ro1$, $Ro2$, $Ro5^*$, $Ro7$, $Br1$ and $Br2$ moves.
\end{theorem}
\begin{proof}
Let $D_1=E_1, \ldots, E_n=D_2$ be a sequence of broken surface diagrams with $E_{i+1}$ related to $E_i$ by a Roseman move for $1 \leq i < n$. Choose $F_i \in {\it cancel}(E_i)$ for $1 \leq i \leq n$. By Lemma \ref{lembranch}\ref{branch4}, $F_{i+1}$ is related to $F_i$ by a sequence of $Ro1$, $Ro2$, $Ro5^*$, $Ro7$, $Br1$ and $Br2$ moves, for $1 \leq i < n$. Since ${\it cancel}(E_1)={\it cancel}(D_1)=\{D_1\}$ and ${\it cancel}(E_n)={\it cancel}(D_2)=\{D_2\}$, we are done.
\end{proof}

\section{Marked graphs and $ab$-surfaces}\label{sectgraph}

Lomonaco \cite{L} and Yoshikawa \cite{Yo} described another method of representing surface links via certain $4$-regular rigid vertex graphs in $\bR^3$. Definition \ref{defgraph} reviews marked graphs and the Yoshikawa moves on marked graph diagrams. Definition \ref{defab} concerns $ab$-surfaces and $ab$-moves. Definition \ref{defthickcap} presents the ${\it thicken}$ and ${\it cap}$ functions. Lemma \ref{lembranch} describes some key properties pertaining to marked graphs, $ab$-surfaces and the ${\it thicken}$ and ${\it cap}$ functions.

\begin{definition}\label{defgraph}
A \textbf{marked graph} $G \subset \bR^3$ is a $4$-regular rigid vertex graph (with some components possibly having no vertices) where:
\begin{itemize}
\item[-] the regions in a rigid neighbourhood of each vertex are given a checkboard coloring, which is usually indicated with a line segment, or marker, as in Figure \ref{figmarker},
\item[-] the links $G_a,G_b \subset \bR^3$ obtained by resolving each marked vertex as in Figure \ref{figGlinks}, are trivial.
\end{itemize}
\begin{figure}
\begin{center}
\includegraphics[scale=.75]{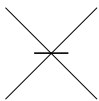}
\caption{A $4$-valent rigid vertex with a marker.}\label{figmarker}
\end{center}
\end{figure}
\begin{figure}
\begin{center}
\includegraphics[scale=.75]{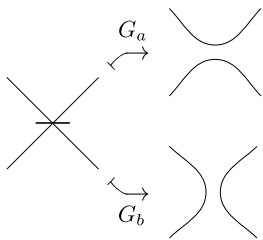}
\caption{The links $G_a,G_b \subset \bR^3$ obtained by resolving all marked vertices in a marked graph $G$.}\label{figGlinks}
\end{center}
\end{figure}
We consider two marked graphs \textbf{equivalent} if they differ by the action of ${\rm Diff}_0^+(\bR^3)$ in a way that preserves rigid neighbourhoods of their vertices, and have agreeing markers. Let $\pM$ be the set of marked graphs. One may project a marked graph to $\bR^2$ and obtain a \textbf{marked graph diagram}. Figure \ref{figyoshimoves} describes the Yoshikawa moves on marked graph diagrams. The ${\it Type \ I}$ moves do not change the equivalence class of the marked graph in $\bR^3$. The results in Kauffman \cite{Kau} show that these first five moves do in fact generate the equivalence relation on marked graph diagrams coming from the action of ${\rm Diff}_0^+(\bR^3)$ on marked graphs in $\bR^3$. The ${\it Type \ II}$ moves differ from ${\it Type \ I}$ moves in that they are defined for marked graphs in $\bR^3$, not just marked graph diagrams in $\bR^2$. \tqed
\begin{figure}
\begin{center}
\includegraphics[scale=1]{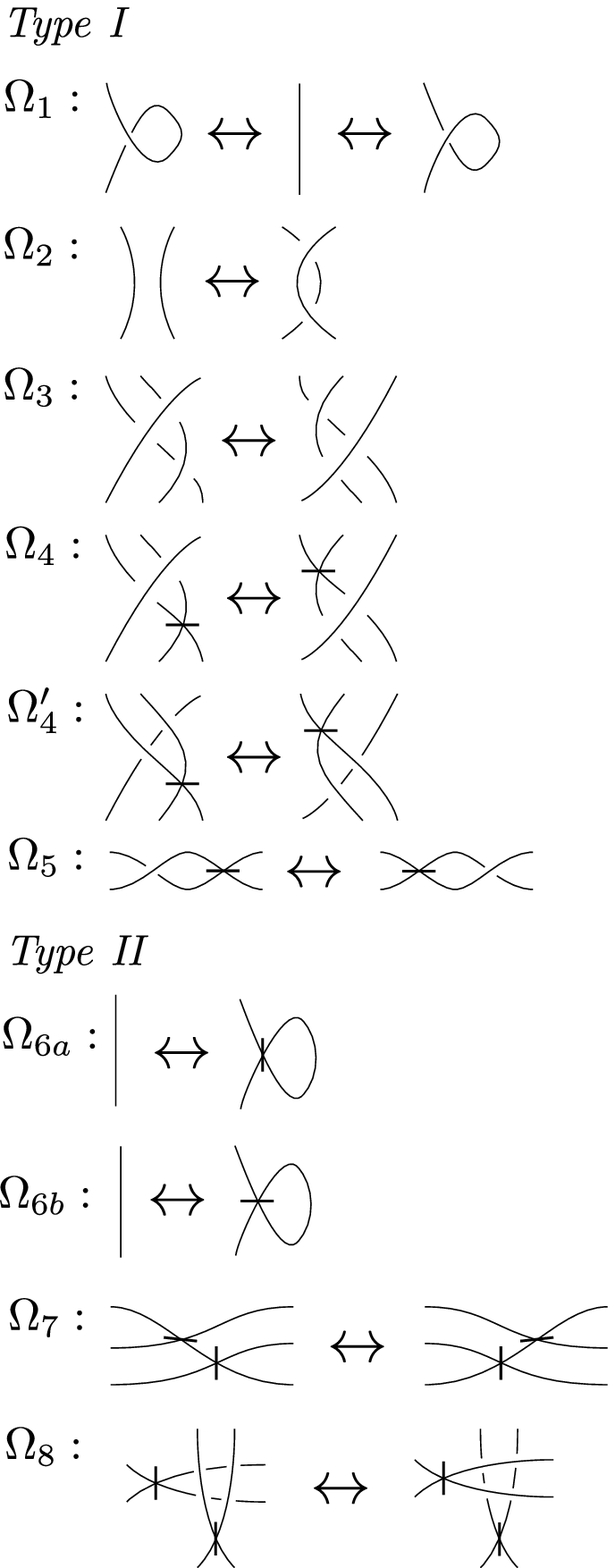}
\caption{Yoshikawa moves on marked graph diagrams.}\label{figyoshimoves}
\end{center}
\end{figure}
\end{definition}

\begin{definition}\label{defab}
An \textbf{$ab$-surface} $R \subset \bR^3$ is a compact not necessarily connected not necessarily orientable surface with boundary where:
\begin{itemize}
\item[-] each boundary component of $R$ is assigned a label from the set $\{a,b\}$,
\item[-] each component of $R$ contains at least one $a$-labelled and at least one $b$-labelled boundary component,
\item[-] the $a$-labelled and $b$-labelled boundary links $\partial_a R, \partial_b R \subset \bR^3$ are trivial.
\end{itemize}
We consider two $ab$-surfaces \textbf{equivalent} if they differ by the action of ${\rm Diff}_0^+(\bR^3)$ and have matching labellings of their boundaries. Let $\pS$ be the set of $ab$-surfaces, $\pS^+$ the subset of orientable $ab$-surfaces, and $\pS_0$ the subset of orientable $ab$-surfaces in which each component has genus $0$. We will refer to the moves on $ab$-surfaces in Figure \ref{figabmoves} as \textbf{$ab$-moves}.
\begin{figure}
\begin{center}
\includegraphics[scale=1]{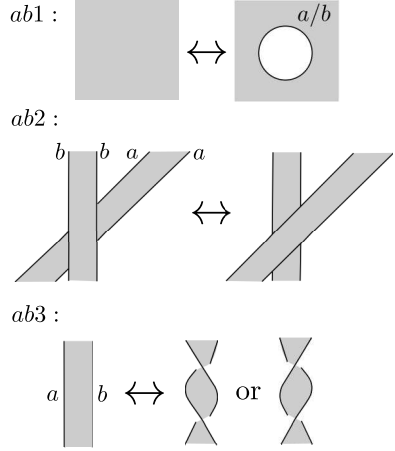}
\caption{$ab$-moves on $ab$-surfaces.}\label{figabmoves}
\end{center}
\end{figure}
\tqed
\end{definition}

\begin{definition}\label{defthickcap}
The function ${\it thicken}:\pM \rightarrow \pS / ab3$, mapping a marked graph to an $ab$-surface up to $ab3$ moves, is given in Figure \ref{figthicken}. Since this is defined locally on the edges and marked vertices of a marked graph, we must glue the final result in a way that preserves the $a/b$-labels coming from marked vertices, and this can only be done up to $ab3$ moves. Note that $G_a=\partial_a ({\it thicken}(G))$ and $G_b=\partial_b ({\it thicken}(G))$ holds for any $4$-regular rigid vertex graph $G$ that satisfies the first condition in Definition \ref{defgraph}. For $R \in \pS$ denote by ${\it thicken}^{-1}(R)$ the set of all $G \in \pM$ with $R \in {\it thicken}(G)$. 

\begin{figure}
\begin{center}
\includegraphics[scale=1]{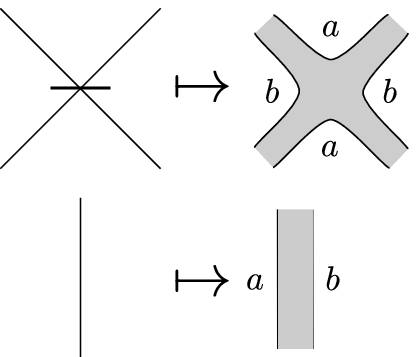}
\caption{The map ${\it thicken}:\pM \rightarrow \pS/ab3$.}\label{figthicken}
\end{center}
\end{figure}

There is a function ${\it cap}: \pS \rightarrow \pL/{\rm Diff}_0^+(\bR^4)$ defined as follows. Given an arbitrary $ab$-surface $R \subset \bR^3$, the boundary of $R$ forms two trivial links $\partial_a R, \partial_b R \subset \bR^3$ (note that we identify $\bR^3$ with $\bR^3 \times \{0\} \subset \bR^4$). Let $\{D_a^i\}_i \subset \bR^3 \times \{1\}$ and $\{D_b^j\}_j \subset \bR^3 \times \{-1\}$ each be a disjoint union of embedded disks with $\pi(\cup_i D_a^i)=\partial_a R$ and $\pi(\cup_j D_b^j) = \partial_b R$ respectively. The union $R \cup ((\partial_a R) \times [0,1]) \cup ((\partial_b R) \times [-1,0])\bigcup_i D_a^i \cup \bigcup_j D_b^j$ is a surface link in $\bR^4$. The isotopy class of this surface link, which we denote ${\it cap}(R)$, is independent of the choice of disk systems $\{D_a^i\}_i$ and $\{D_b^j\}_j$ in $\bR^3\times \{-1,1\}$, see Kamada \cite[Proposition 8.6]{Kam}. Due to this definition, it is reasonable to think of the labels $a$ and $b$ as shorthand for ``above" and ``below" (with respect to the $u$-coordinate).

Given a marked graph $G \in \pM$, the surface link isotopy class associated to $G$ is the isotopy class of ${\it cap}(R)$ for any $R \in {\it thicken}(G)$. We will see in Lemma \ref{lemgraphab}\ref{graphab4} that the ${\it cap}$ function is invariant under all $ab$-moves, so in particular the isotopy class of ${\it cap}(R)$ is the same for all $R \in {\it thicken}(G)$, since all such $R$ are related by $ab3$ moves.
\tqed
\end{definition}

\begin{lemma}\label{lemgraphab}
\begin{enumerate}[a)]
\item\label{graphab1} If $R \in \pS_0$ then ${\it thicken}^{-1}(R)$ is non-empty and any two graphs in ${\it thicken}^{-1}(R)$ are related by $\Omega_7$ moves.
\item\label{graphab2} If $G_1,G_2 \in \pM$ are related by ${\it Type \ II}$ Yoshikawa moves and $R_1 \in {\it thicken}(G_1)$ and $R_2 \in {\it thicken}(G_2)$ then $R_1$ and $R_2$ are related by $ab$-moves.
\item\label{graphab3} If $R_1,R_2 \in \pS_0$ are related by $ab$-moves and $G_1 \in {\it thicken}^{-1}(R_1)$ and $G_2 \in {\it thicken}^{-1}(R_2)$ then $G_1$ and $G_2$ are related by ${\it Type \ II}$ Yoshikawa moves.
\item\label{graphab4} If $R_1,R_2 \in \pS$ are related by $ab$-moves then ${\it cap}(R_1)={\it cap}(R_2)$.
\end{enumerate}
\end{lemma}
\begin{proof}
\begin{enumerate}[a)]
\item If $G \in {\it thicken}^{-1}(R)$ then $G$ is equivalent to a graph embedded in $R$. Thus there is no loss in generality restricting ourselves to marked graphs $G$ with $G \subset R$. There is also no loss in generality in assuming $R$ is connected.

Assume $R$ has $a$-boundary components $a_1,\ldots, a_v$ and $b$-boundary components $b_1,\ldots,b_f$ for $v,f \geq 1$. Let $c_1,\ldots,c_v \subset R$ be disjoint simple closed curves with $c_i$ parallel to $a_i$. Let $p_1,\ldots, p_{v+f-2} \subset R$ be a collection of disjoint simple arcs, whose interiors are disjoint from all curves $c_i$, with $\partial p_j \subset \bigcup_i c_i$ and such that each component of $R \setminus \left(\bigcup_i c_i \cup \bigcup_j p_j\right)$ is an annulus containing one unlabelled boundary component and one labelled boundary component. One may view the curves $c_i$ and arcs $p_j$ as specifying a planar graph, with vertices the curves $c_i$, faces the $b$-labelled boundary components, and edges the arcs $p_j$.

We call such a collection of simple curves and arcs up to isotopies in $R$ (i.e. up to the action of ${\rm Diff}_0^+(R)$), an \textbf{$a$-system}. We may construct a marked graph $G \in {\it thicken}^{-1}(R)$ with $G \subset R$ from an $a$-system via the transformation in Figure \ref{figmakegraph}.
\begin{figure}
\begin{center}
\includegraphics[scale=1]{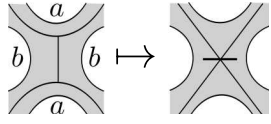}
\caption{Constructing a marked graph in ${\it thicken}^{-1}(R)$ from an $a$-system in Lemma \ref{lemgraphab}\ref{graphab1}.}\label{figmakegraph}
\end{center}
\end{figure}
Conversely, given a marked graph $G \subset R$, the inverse of the operation in Figure \ref{figmakegraph}, creates an $a$-system. The $\Omega_7$ move on marked graphs in $R$ corresponds to the slide move in Figure \ref{figsysmoves} on $a$-systems.
\begin{figure}
\begin{center}
\includegraphics[scale=1]{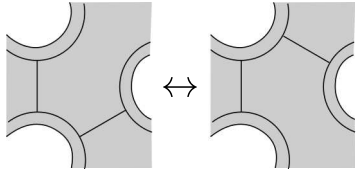}
\caption{The slide move for $a$-systems in $R$ in Lemma \ref{lemgraphab}\ref{graphab1}.}\label{figsysmoves}
\end{center}
\end{figure}
One can use slide moves to bring any $a$-system into a form where there exist two open intervals $U,V \subset c_1$ with $U \cap V = \emptyset$ such that $v-1$ of the paths have one endpoint in $U$ and one endpoint on $c_i$ for $1< i \leq v$ and the remaining $f-1$ paths have both endpoints adjacent to each other in $V$. The claim that any two such $a$-systems are related by slide moves can be deduced from Kamada \cite[Proposition 2.14]{Kam}.

Note that the set of $a$-systems up to the action of ${\rm Diff}_0^+(R)$ is infinite, if $R$ has four or more boundary components. However in this case ${\it thicken}^{-1}(R)$ may still be finite if $R$ admits many symmetries in $\bR^3$, say if $R \subset \bR^2$ (in which case all of the $a$-systems convert into one of finitely many marked graphs up to the action of ${\rm Diff}_0^+(\bR^3)$).

\item If $G_1$ is related to $G_2$ by an $\Omega_{6a}$ or $\Omega_{6b}$ move then $R_1$ and $R_2$ are related by $ab1$ and $ab3$ moves. If $G_1$ and $G_2$ are related by an $\Omega_7$ move then ${\it thicken}(G_1)={\it thicken}(G_2)$ and $R_1$ and $R_2$ are related by $ab3$ moves. If $G_1$ and $G_2$ are related by an $\Omega_8$ move then $R_1$ and $R_2$ are related by $ab2$ and $ab3$ moves.

\item We consider each $ab$-move separately. Assume an $\overrightarrow{ab1}$ move takes $R_1$ to $R_2$ and creates an $a$-labelled boundary component. Then one may obtain an $a$-system for $R_2$ from any $a$-system for $R_1$ by connecting an extra path to the new simple closed curve parallel to the new $a$-labelled boundary component. This corresponds to an $\overrightarrow{\Omega_{6a}}$ move. Thus $G_1$ and $G_2$ are related by $\Omega_7$ and $\Omega_{6a}$ moves. If the $ab1$ move involves a $b$-labelled boundary component the same reasoning may be used except with the dual notion of $b$-systems, and $G_1$ and $G_2$ will be related by $\Omega_7$ and $\Omega_{6b}$ moves.

Assume now $R_1$ and $R_2$ are related by an $ab2$ move. Any $ab2$ move determines two disjoint simple paths $p,q \subset R_1$ with $\partial p \subset \partial_a R_1$ and $\partial q \subset \partial_b R_1$. If either path cobounds with a segment of $\partial R_1$ an embedded disk in $R_1$, then the $ab2$ move can be realized as an equivalence of $ab$-surfaces in $\bR^3$. Thus assume neither path cobounds such a disk. First consider the case where $p,q$ are in different components of $R_1$. One can find an $a$-system in the component containing $p$ such that $p$ is one of the paths of the $a$-system and similarly one can find a $b$-system in the component containing $q$ such that $q$ is one of the paths of the $b$-system. By finding $a$ or $b$-systems in the remaining components of $R_1$, we obtain a marked graph $G_1' \in {\it thicken}(R_1)$ for which the $ab2$ move corresponds to an $\Omega_8$ move taking $G_1'$ to some $G_2' \in {\it thicken}(R_2)$. Thus $G_1$ and $G_2$ are related by $\Omega_7$ and $\Omega_8$ moves. Now consider the case where $p,q$ are in the same component of $R_1$. Since we assumed neither $p,q$ cobounds with $\partial R_1$ an embedded disk in $R_1$, the componet of $R_1$ containing $p,q$ must have at least two $a$-labelled and two $b$-labelled boundary components. We must find an $a$-system for this component of $R_1$, such that $p$ is one of the paths of the $a$-system and $q$ is one of the paths of the dual $b$-system (any $a$-system induces a unique dual $b$-system and vice versa). This amounts to finding an $a$-system for which $p$ is one of the paths of the $a$-system and $q$ transversely intersects some other path in this $a$-system (not $p$) exactly once. It is not difficult to see that any such path along with $p$, can be extended to an $a$-system so we are done. As before, we obtain a marked graph $G_1' \in {\it thicken}(R_1)$ for which the $ab2$ move corresponds to an $\Omega_8$ move taking $G_1'$ to some $G_2' \in {\it thicken}(R_2)$. Thus $G_1$ and $G_2$ are related by $\Omega_7$ and $\Omega_8$ moves.

Assume finally that $R_1$ and $R_2$ are related by an $ab3$ move. The $ab3$ move determines a simple compact interval $p$ in $R_1$ with one endpoint in $\partial_a R_1$ and the other in $\partial_b R_1$. We can readily find an $a$-system in $R_1$ for which all paths in the system are disjoint from $p$. This $a$-system gives rise to a marked graph $G_0$ for which $R_1,R_2 \in {\it thicken}(G_0)$. Thus $G_1$ and $G_2$ are related to $G_0$ by $\Omega_7$ moves and hence are related to each other by $\Omega_7$ moves. 

\item For the $ab1$ move this should be clear. For the $ab2$ and $ab3$ moves, the isotopies in $\bR^4$ connecting ${\it cap}(R_1)$ and ${\it cap}(R_2)$ are given in Figures \ref{figmovab2} and \ref{figmovab3}.
\begin{figure}
\begin{center}
\includegraphics[scale=1]{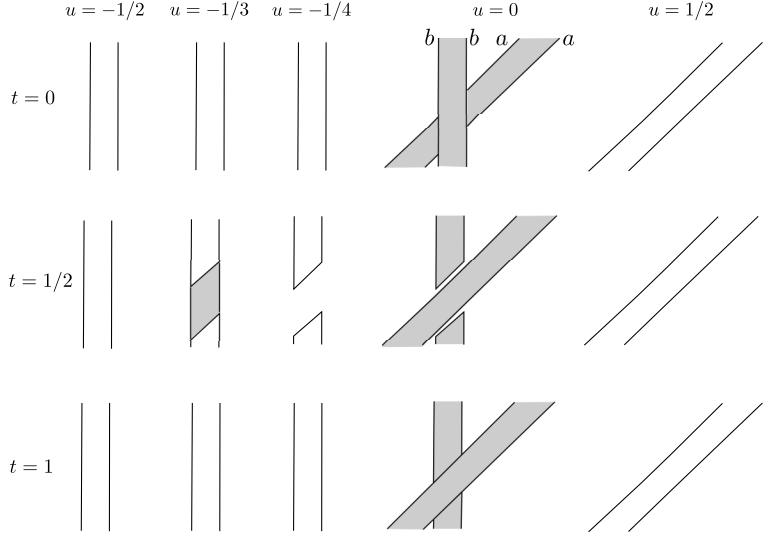}
\caption{An isotopy between ${\it cap}(R_1)$ and ${\it cap}(R_2)$ when $R_1$ and $R_2$ are related by an $ab2$ move.}\label{figmovab2}
\end{center}
\end{figure}
\begin{figure}
\begin{center}
\includegraphics[scale=1]{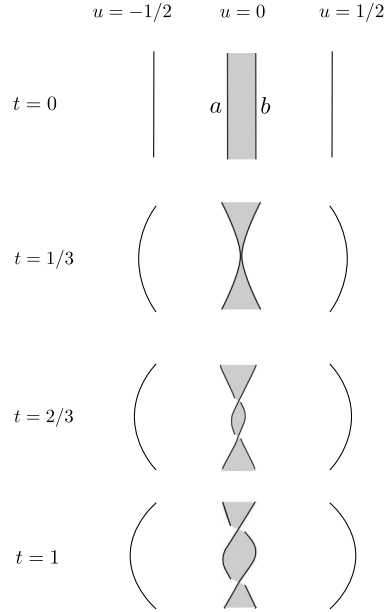}
\caption{An isotopy between ${\it cap}(R_1)$ and ${\it cap}(R_2)$ when $R_1$ and $R_2$ are related by an $ab3$ move.}\label{figmovab3}
\end{center}
\end{figure}
\end{enumerate}
\end{proof}

\begin{remark}
Lemma \ref{lemgraphab} might generalize to $ab$-surfaces in $\pS^+$, without restricting to $\pS_0$ in some instances as we have done.
\end{remark}

\section{A map from broken surface diagrams to $ab$-surfaces}\label{sectproof}

In Definition \ref{defperf} we describe the function ${\it perforate}$, assigning to any branch-free broken surface diagram an $ab$-surface. In Lemma \ref{lemperf} we prove some properties of the ${\it perforate}$ function and observe how it behaves when $Ro1$, $Ro2$, $Ro5^*$, $Ro7$, $Br1$ and $Br2$ moves are performed on the input. In Theorem \ref{thmmain} we prove the main result, that two marked graphs representing isotopic $2$-links are related by a sequence of ${\it Type \ II}$ Yoshikawa moves.

\begin{figure}
\begin{center}
\includegraphics[scale=1]{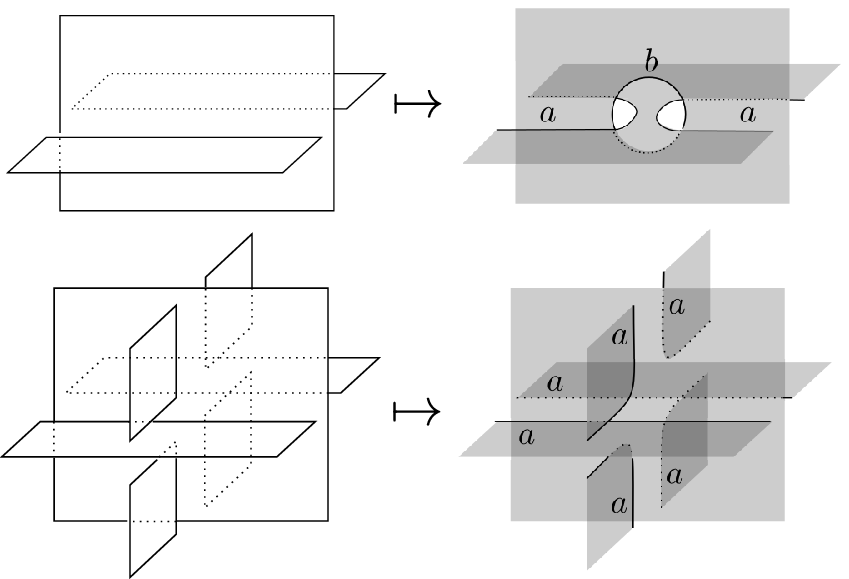}
\caption{The function ${\it perforate}:\pB^b(\pL) \rightarrow \pS$.}\label{figperf}
\end{center}
\end{figure}

\begin{definition}\label{defperf}
We define the function ${\it perforate}:\pB^b(\pL) \rightarrow \pS$ by the transformations in Figure \ref{figperf}, with a caveat. If the surface obtained via Figure \ref{figperf} has no $a$-labelled (resp. $b$-labelled) boundary components, we add arbitrarily a small $a$-labelled (resp. $b$-labelled) boundary component to fulfil the second property of Definition \ref{defab}. However we will often not bother to draw these extra components.
\tqed
\end{definition}

\begin{figure}
\begin{center}
\includegraphics[scale=1]{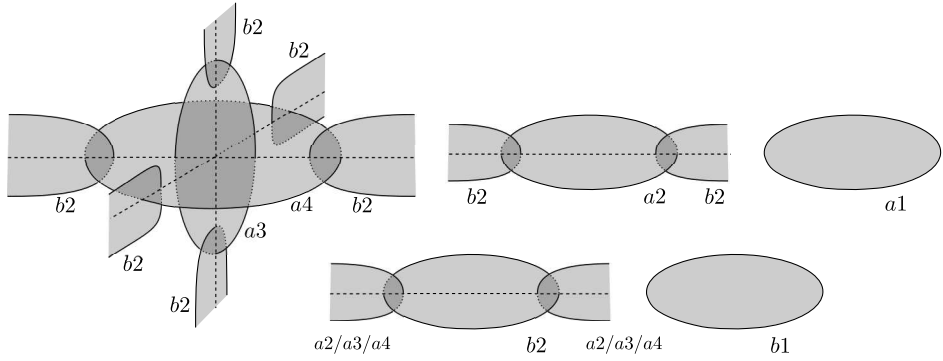}
\caption{Types of disks in the proof of Lemma \ref{lemperf}\ref{perf1} and \ref{lemperf}\ref{perf2}. Note that disks of type $a3$ and $a4$ are dependent and can only occur together.}\label{figtypes}
\end{center}
\end{figure}

\begin{remark}\label{remperf}
If $L \in \pL$ is such that $\pi(L) \in \pB^b(\pL)$, the set ${\it sing}(\pi(L)) \subset \bR^3$ is an embedded $6$-regular graph, possibly with some components having no vertices. Each edge or closed component gives rise to one $b$-labelled boundary component in ${\it perforate}(\pi(L))$. The surface ${\it perforate}(\pi(L))$ is indeed an $ab$-surface. The $b$-labelled boundary $\partial_b({\it perforate}(\pi(L)))$ forms an unlink and the existence of a $b$-labelled boundary component at each edge forces $\partial_a({\it perforate}(\pi(L)))$ to form an unlink as well. The isotopy class ${\it cap}({\it perforate}(\pi(L)))$ agrees with the isotopy class of $L$. Note also that the $ab$-surface ${\it perforate}(D)$ for any $D \in \pB^b(\pL)$ induces the $0$-framing on each of its boundary components.
\tqed
\end{remark}

\begin{lemma}\label{lemperf}
\begin{enumerate}[a)]
\item\label{perf1} If $R \in \pS$ and there exist systems $\{D_a^i\}_i$ and $\{D_b^j\}_j$ of disks in $\bR^3$ such that:
\begin{itemize}
\item[-] $\cup_i D_a^i=\partial_a R$ and $\cup_j D_b^j=\partial_b R$,
\item[-] each disk is embedded and is of one of the types in Figure \ref{figtypes}, based on its intersections with other disks and $R$ (which is not depicted in the figure),
\end{itemize}
then $R$ is related by $ab$-moves to a surface of the form ${\it perforate}(D)$ for some $D \in \pB^b(\pL)$.
\item\label{perf2} If $R \in \pS_0$ then $R$ is related by $ab$-moves to an $ab$-surface of the form ${\it perforate}(D)$ for some $D \in \pB^b(\pL_0)$.
\item\label{perf3} If two broken surface diagrams $D,D' \in \pB^b(\pL)$ are related by an $Ro1$, $Ro2$, $Ro5^*$, $Ro7$, $Br1$ or $Br2$ move then ${\it perforate}(D)$ and ${\it perforate}(D')$ are related by $ab$-moves.
\end{enumerate}
\end{lemma}
\begin{proof}
\begin{enumerate}[a)]
\item Each boundary component of the $ab$-surface $R$ is an unknot and has an induced framing from the embedding $R \subset \bR^3$. Due to the way $R$ intersects the described systems of disks, it must be the case that $R$ induces the $0$-framing on each of its boundary components. The union $D=R \cup \bigcup_i D_a^i \cup \bigcup_j D_b^j$ is a generic projection with no branch points. The surface $R$ is nearly in the form ${\it perforate}(D)$, we must only eliminate all boundary components bounding disks of type $a1$, $b1$ or $a2$. If there are disks of type $a1$ or $b1$, then those boundary components of $R$ should be eliminated with $\overleftarrow{ab1}$ moves. If there are disks of type $a2$ then those boundary components may be eliminated with $ab2$ and $\overleftarrow{ab1}$ moves. Once this has been done, we have $R={\it perforate}(D)$.

\item First we show that we can use $ab3$ moves to ensure each boundary component of $R$ is $0$-framed. It is enough to prove this when $R$ is connected. By the second property of Definition \ref{defab}, $R$ has at least one $a$-labelled and one $b$-labelled boundary component. We can use $ab3$ moves to ensure that all framings are $0$ except for one $a$-labelled boundary component. By Lemma \ref{lemgraphab}\ref{graphab1} there is $G \in \pM$ with $G \subset R$ and $thicken(G)=R$. The graph $G$ is planar, and we can unknot it to some plane graph in $\bR^2$, while preserving the framings. In this form it is easy to see that this final $a$-labelled boundary component must automatically be $0$-framed, thus it must be $0$-framed in $R$ as well.

Let $\{E_a^i\}_i \subset \bR^3 \times \{1\}$ and $\{E_b^j\}_j \subset \bR^3 \times \{-1\}$ be systems of disks as in the definition of the ${\it cap}$ function, for the surface $R$. Let $D_a^i=\pi(E_a^i)$ and $D_b^j=\pi(E_b^j)$. Note that $D_a=\cup_i D_a^i$ is a disjoint union of embedded disks in $\bR^3$, as is $D_b=\cup_j D_b^j$. Since each component of $\partial R$ is $0$-framed, we may assume that $R \cap D_a$ and $R \cap D_b$ are contained in the interiors of $D_a$ and $D_b$. We perturb $D_a$ and $D_b$ so that $R \cap D_a$, $R \cap D_b$ and $D_a \cap D_b$ are transverse in $\bR^3$. We also assume that $\left(R \cap D_a\right) \cap \left(D_a \cap D_b\right)$ is transverse in $D_a$, etc. Now with such a system of disks, $R \cup D_a \cup D_b \subset \bR^3$ is a generic projection of ${\it cap}(R)$ with no branch points. Ultimately, we would like to find systems of disks as in Lemma \ref{lemperf}\ref{perf1}.
\begin{figure}
\begin{center}
\includegraphics[scale=1]{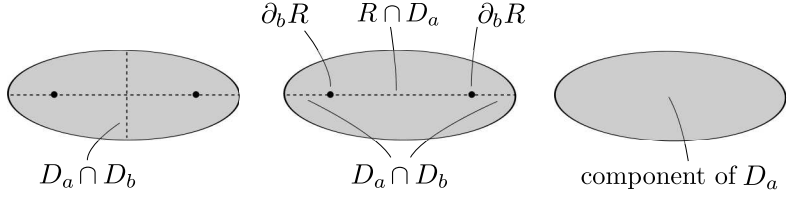}
\caption{Three desirable possibilities for a component of $D_a$ intersecting $R \cup D_b$ in Lemma \ref{lemperf}\ref{perf1}.}\label{figDa}
\end{center}
\end{figure}

Our next goal is to make a series of adjustments to $R$, so that each component of $D_a$ intersects $R \cup D_b$ in one of the three ways in Figure \ref{figDa}. We may use $\overrightarrow{ab1}$ moves to add extra $b$-labelled boundary components so that $R \cap D_a$ contains no closed curves. Specifically, for each closed curve in this intersection we create a $b$-labelled boundary component in a neighbourhood of some point on the curve. The closed curve then is replaced with a compact interval, as for example in the first row of Figure \ref{figadjust}. We may use $\overrightarrow{ab1}$ moves to add extra $a$-labelled boundary components, followed by $ab2$ moves so that each disk in $D_a$ intersects $R$ in exactly one compact interval, as in the second row of Figure \ref{figadjust}. We make further adjustments by shrinking each boundary component in $\partial_a R=\partial D_a$ small enough (via an equivalence in $\bR^3$ of the surface $R$), so that each component of $D_a$ intersects $R \cup D_b$ as in the top-left of Figure \ref{figadjust2}, possibly with more transverse intersections from $\left(R \cap D_a\right) \cap \left(D_a \cap D_b\right)$ (the figure depicts two such intersections). We then may use a series of $ab1$ and $ab2$ moves as in Figure \ref{figadjust2}, to achieve our stated goal of having each component of $D_a$ intersect $R \cup D_b$ in one of the three ways in Figure \ref{figDa}.
\begin{figure}
\begin{center}
\includegraphics[scale=1]{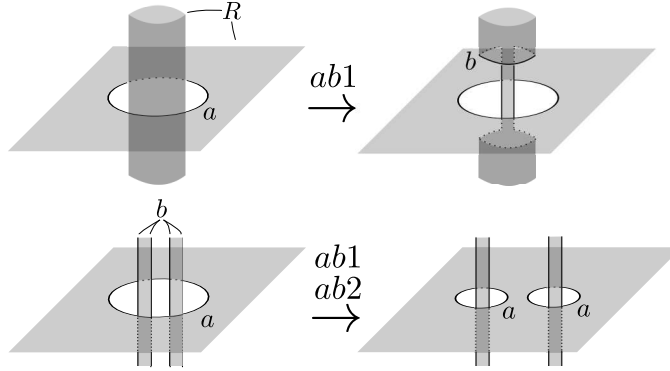}
\caption{Adjusting $R$ in the proof of Lemma \ref{lemperf}\ref{perf1}.}\label{figadjust}
\end{center}
\end{figure}
\begin{figure}
\begin{center}
\includegraphics[scale=1]{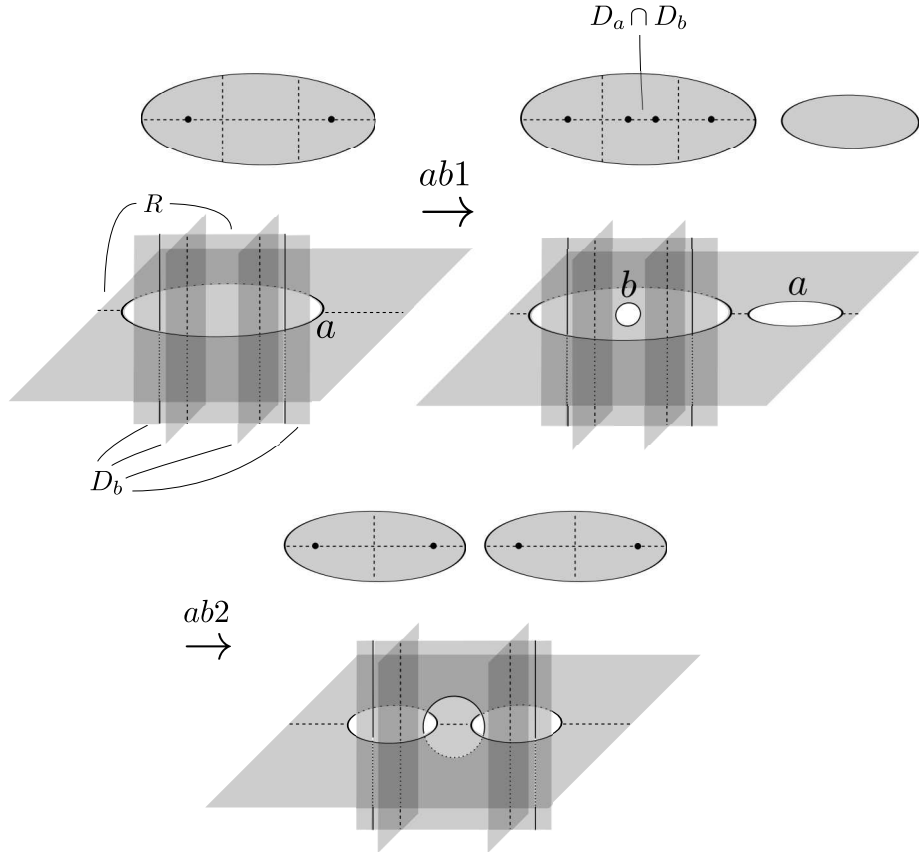}
\caption{Further adjustments of $R$ in Lemma \ref{lemperf}\ref{perf1}, so that each component of $D_a$ intersects $R \cup D_b$ in one of the three ways in Figure \ref{figDa}.}\label{figadjust2}
\end{center}
\end{figure}

The union $R \cup D_a \cup D_b$ remains a generic projection with no branch points, and the disk systems $D_a$ and $D_b$ remain embedded in $\bR^3$. Each triple point of $R \cup D_a \cup D_b$ arises from a component of $D_a$ as in the top-left of Figure \ref{figDa}. We perform $ab1$ moves to add an $a$-labelled boundary component at each triple point, as in Figure \ref{figadjust3}. By now, for each $j$ we have that $R \cap D_b^j$ is a disjoint union of simple closed curves and compact intervals with endpoints in $\partial_a R$. Each compact interval in $R \cap D_b^j$ is properly contained in an edge of ${\it sing}(R \cup D_a \cup D_b)$ and each simple closed curve in $R \cap D_b^j$ is disjoint from any triple point in ${\it sing}(R \cup D_a \cup D_b)$ and remains a split simple closed curve in ${\it sing}(R \cup D_a \cup D_b)$. We repeat the steps of Figure \ref{figadjust} for the disks $D_b^j$ with the roles of $a$ and $b$ reversed. At this point, the systems $\{D_a^i\}_i$ and $\{D_b^j\}_j$ of disks satisfy the conditions of Lemma \ref{lemperf}\ref{perf1} and we are done.

\begin{figure}
\begin{center}
\includegraphics[scale=1]{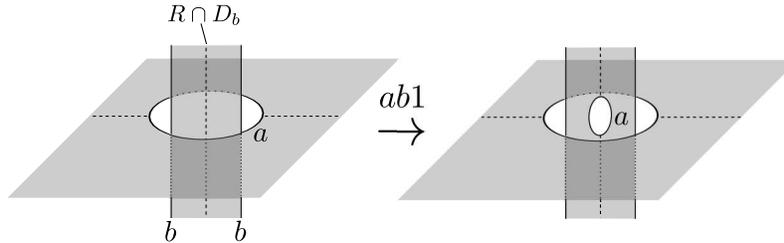}
\caption{Adding an $a$-labelled boundary component at a triple point.}\label{figadjust3}
\end{center}
\end{figure}

\item We must check that each of the mentioned moves on branch-free broken surface diagrams can be realized by $ab$-moves on their images under the ${\it perforate}$ function. See Figures \ref{figcheckro1}, \ref{figcheckro2}, \ref{figcheckro5} and \ref{figcheckbr2} for the $Ro1$, $Ro2$, $Ro5^*$ and $Br2$ moves. The reader should verify that the transitions in each figure are achievable by $ab$-moves and equivalences in $\bR^3$. For the $Ro7$ move, Figure \ref{figcheckro7} shows the necessary changes to the bottommost sheet, which will have only $b$-labelled boundary components. One may use $ab1$ and $ab2$ moves to merge all $b$-labelled boundary components into a single boundary component, the sheet can then be pushed via an equivalence in $\bR^3$ past the triple point (or where it normally would be), and then the $b$-labelled boundary components can be restored with $ab1$ and $ab2$ moves. In the figure, the dashed line indicates where the other three sheets would intersect the bottommost sheet in the broken surface diagram.

Instead of the $Br2$ move, in Figure \ref{figcheckbr1} we check a move that is equivalent to $Br2$ modulo $Ro1$, $Ro2$, $Ro5^*$, $Ro7$ and $Br1$ moves.
\end{enumerate}
\end{proof}

\begin{figure}
\begin{center}
\includegraphics[scale=1]{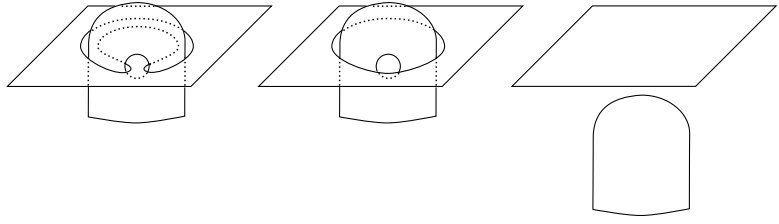}
\caption{Checking the $Ro1$ move in Lemma \ref{lemperf}\ref{perf3}.}\label{figcheckro1}
\end{center}
\end{figure}
\begin{figure}
\begin{center}
\includegraphics[scale=1]{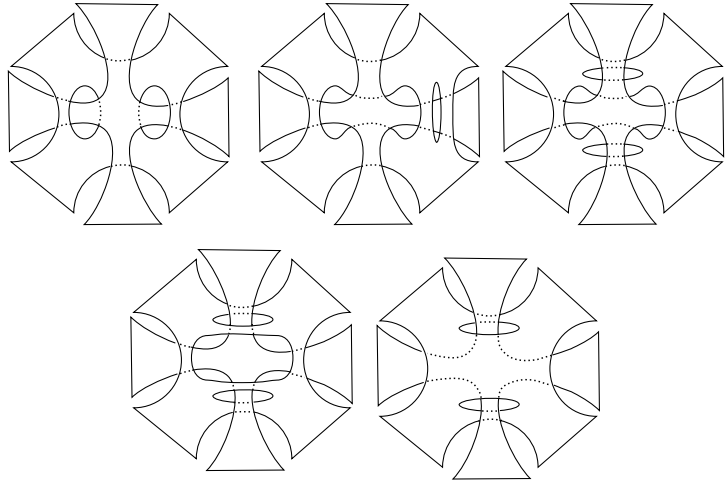}
\caption{Checking the $Ro2$ move in Lemma \ref{lemperf}\ref{perf3}.}\label{figcheckro2}
\end{center}
\end{figure} 

\begin{figure}
\begin{center}
\includegraphics[scale=1]{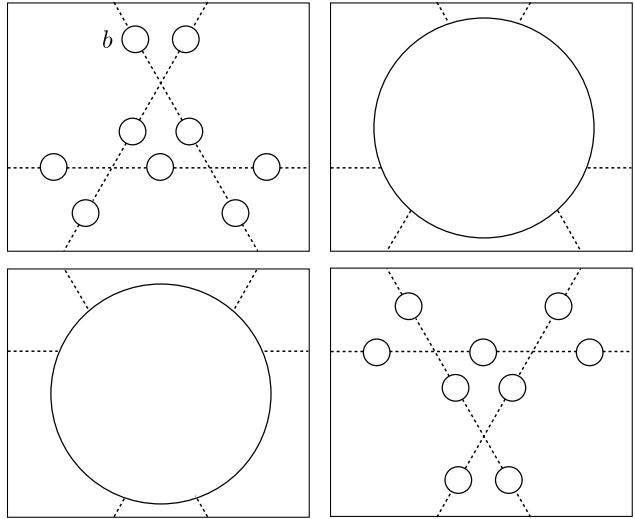}
\caption{Checking the $Ro7$ move in Lemma \ref{lemperf}\ref{perf3}.}\label{figcheckro7}
\end{center}
\end{figure}

\begin{figure}
\begin{center}
\includegraphics[scale=1]{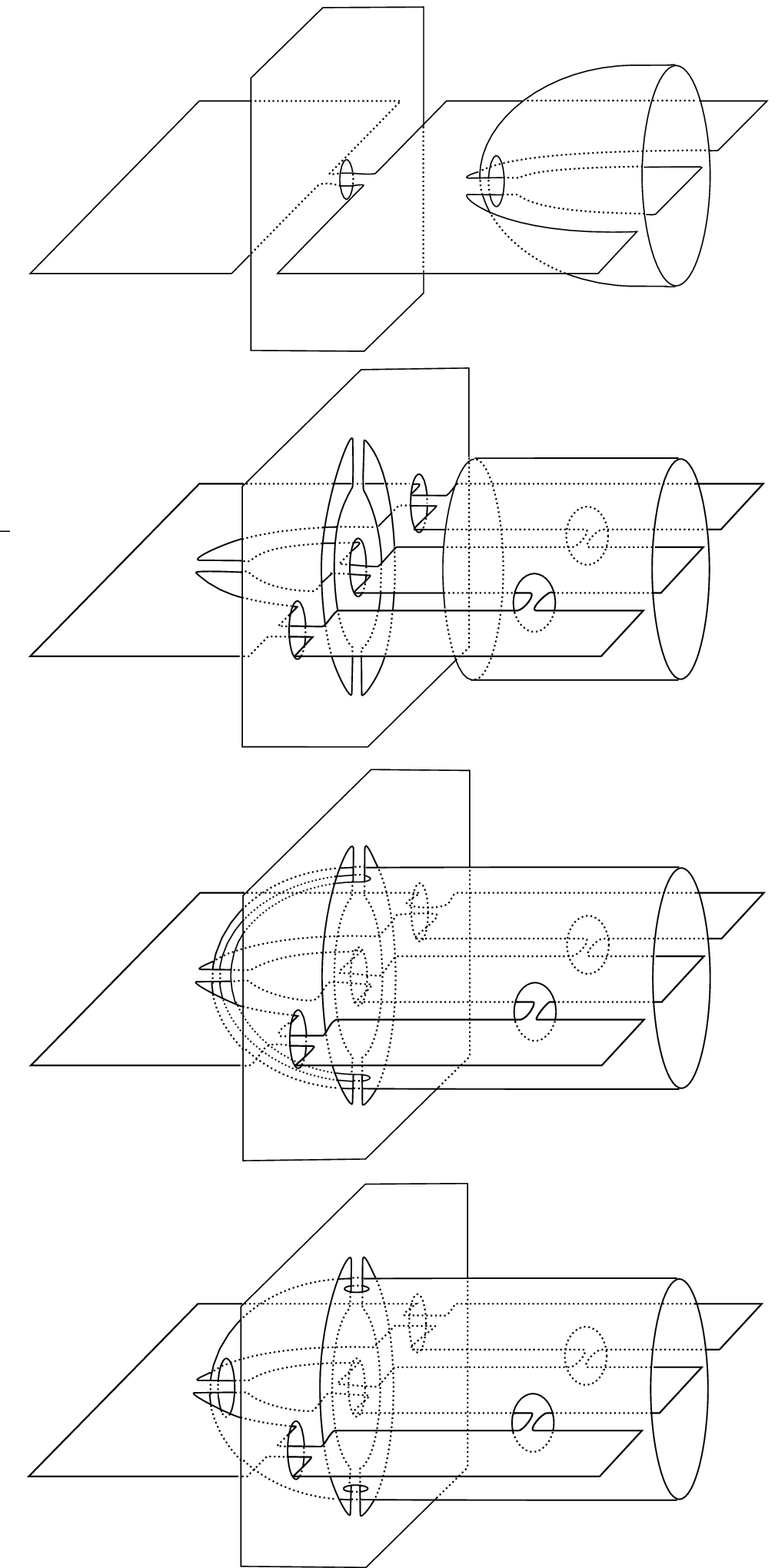}
\caption{Checking the $Ro5^*$ move in Lemma \ref{lemperf}\ref{perf3}.}\label{figcheckro5}
\end{center}
\end{figure}

\begin{figure}
\begin{center}
\includegraphics[scale=1]{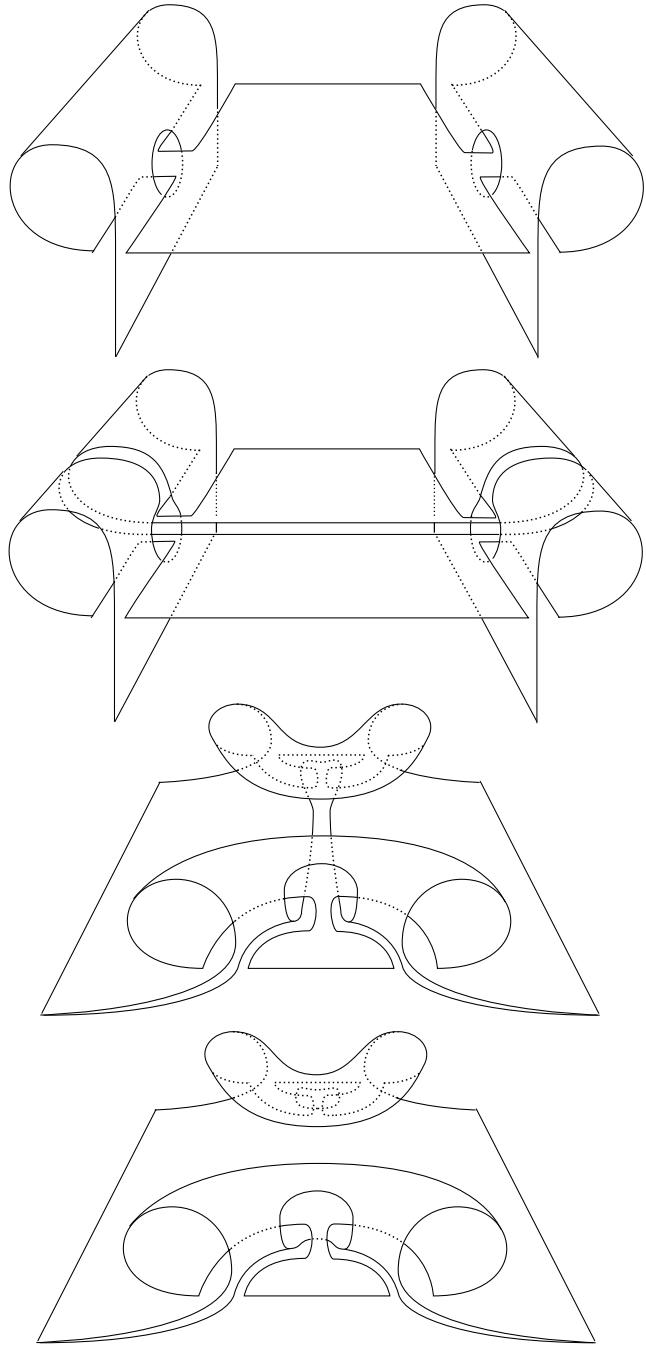}
\caption{Checking the $Br1$ move in Lemma \ref{lemperf}\ref{perf3}.}\label{figcheckbr1}
\end{center}
\end{figure}
\begin{figure}

\begin{center}
\includegraphics[scale=1]{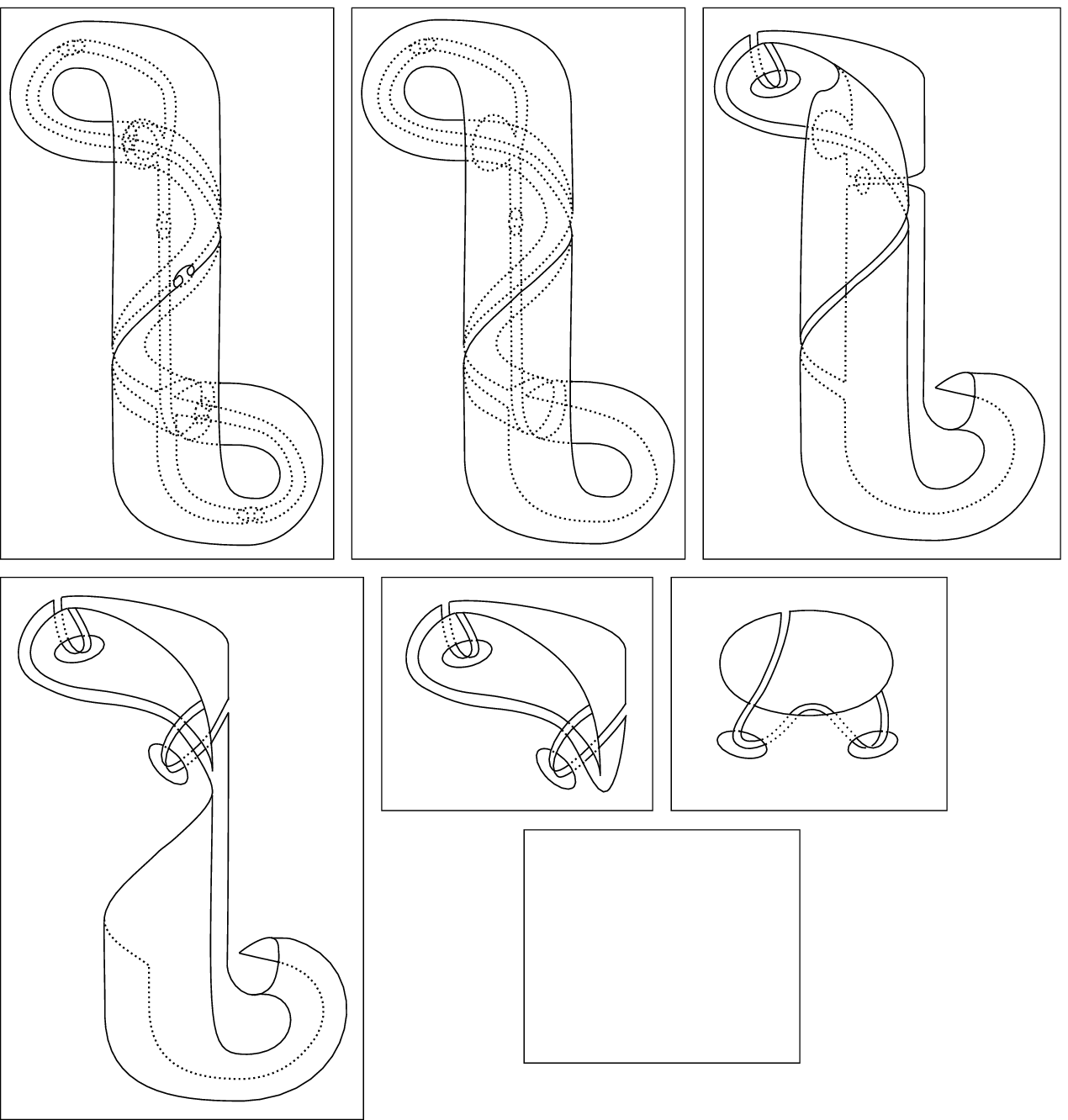}
\caption{Checking a move equivalent to the $Br2$ move in Lemma \ref{lemperf}\ref{perf3}.}\label{figcheckbr2}
\end{center}
\end{figure}

\begin{theorem}\label{thmmain}
If $G_1,G_2 \in \pM$ are such that ${\it cap}({\it thicken}(G_1))$ and ${\it cap}({\it thicken}(G_2))$ are isotopic $2$-links, then $G_1$ and $G_2$ are related by a sequence of ${\it Type \ II}$ Yoshikawa moves.
\end{theorem}
\begin{proof}
Let $R_1 \in {\it thicken}(G_1)$ and $R_2 \in {\it thicken}(G_2)$. By Lemma \ref{lemperf}\ref{perf2} there exists an $ab$-surface $R_i'$ related by $ab$-moves to $R_i$ so that $R_i'$ is of the form ${\it perforate}(D_i)$ for some $D_i \in \pB^b(\pL_0)$ for $i=1,2$. Since $D_1$ represents a $2$-link isotopic to the $2$-link represented by $D_2$, by Roseman's theorem there is a sequence of Roseman moves from $D_1$ to $D_2$. By Theorem \ref{thmbranch} there is a sequence of $Ro1$, $Ro2$, $Ro5^*$, $Ro7$, $Br1$ and $Br2$ moves taking $D_1$ to $D_2$. By Lemma \ref{lemperf}\ref{perf3} there is a sequence of $ab$-moves taking $R_1'$ to $R_2'$. Thus there is a sequence of $ab$-moves taking $R_1$ to $R_2$. By Lemma \ref{lemgraphab}\ref{graphab3} there is a sequence of ${\it Type \ II}$ Yoshikawa moves taking $G_1$ to $G_2$.

Note also that if we are presented with marked graph diagrams of $G_1$ and $G_2$, then the marked graph diagrams are related by a sequence of ${\it Type \ I}$ and ${\it Type \ II}$ Yoshikawa moves, by the above and the results of Kauffman \cite{Kau} for diagrams of $4$-regular rigid vertex spatial graphs.
\end{proof}

\end{document}